\documentclass{amsart}
\usepackage{amsmath,amscd,amstext,amsthm,amsfonts,latexsym, color}
\usepackage[dvips]{graphicx}
\usepackage[english]{babel}

\newtheorem{theorem}{Theorem}

\newtheorem{lemma}[theorem]{Lemma}
\newtheorem{proposition}[theorem]{Proposition}

\newtheorem{definition}[theorem]{Definition}
\newtheorem{remark}[theorem]{Remark}
\newtheorem{example}[theorem]{Example}

\begin{document}




 \email{baravi@mat.ufrgs.br}
 \email{oliveira.elismar@gmail.com}
 \email{fagnerbernardini@gmail.com}

\thanks{The first author is partially supported by CNPq, CAPES and FAPERGS}



\bigskip

\title[A note on the dynamics of linear automorphisms]{A note on the dynamics of linear automorphisms of a measure convolution algebra}

\date{\today}
\maketitle
\centerline{\scshape A. Baraviera, E. Oliveira and F. B. Rodrigues}
\medskip
{\footnotesize
 \centerline{Instituto de Matem\'atica-UFRGS}
   \centerline{ Avenida Bento Gon\c{c}alves 9500 Porto Alegre-RS Brazil}
} 

\begin{abstract} In this work we are going to study the dynamics of the linear automorphisms of a measure convolution algebra over a finite group, $T(\mu)=\nu * \mu$. In order to understand an classify the asymptotic behavior of this dynamical system we provide an alternative to classical results, a very direct way to understand  convergence of the sequence $\{\nu^{n}\}_{n\in\mathbb{N}}$,
where $G$ is a finite group, $\nu\in\mathcal{P}(G)$ and $\nu^n=\underbrace{\nu\ast...\ast\nu}_n$, trough the subgroup generated by his support.
\end{abstract}




\section{Introduction}
    The space of probabilities on a metric space $G$
    (or more generally, Radon Measures) has two natural classes of linear automorphisms, the first one is the push forward (induced by some fixed map $f \colon G \to G$),
    that has been extensively studied by  Sigmund (in \cite{K. Bauer}) and Komuro (in \cite{Komuro}); more recently it also appears  in
    Kloeckner (in \cite{kloeckner}) for example, and just take in consideration the linear structure of the space of measures. The second one, when $G$ is
    a topological group, is based on the convolution of two measures.
    In this case the space of Radon measures is an infinite
    dimensional Banach algebra, with respect to the convolution operation,
    that is, a Measure Convolution Algebra  (see \cite{Bob} pg 73 and \cite{J.Wermer}).
    Hence the other natural linear automorphism is $T(\mu)=\nu * \mu$,
    for a fixed measure $\nu$.
    In this way we propose to understand the topological dynamics
    of this map. The iteration of $T$ lead us to analyze the powers of convolutions of $\nu$,
    since from basic properties of the operation $*$ we get
    that iterating $n$ times the map $T_{\nu}$ is the convolution $\nu^{n} * \mu$.



     The problem of study powers of convolution of probability measures has been studied in several papers in the last few years and has several applications in statistics and group theory (see \cite{HogM} and \cite{Chakraborty}). In a general setting $G$ is a compact topological group,
     $\mathcal{P}(G)$ is the set of all probability measures on $G$ and $\nu\in\mathcal{P}(G)$.


     The the main goal of this paper is to establish direct conditions on the support of measure $\nu$, what is quite natural from the ergodic point of view,  to  ensure convergence of  the sequence  $\{\nu^n:=\underbrace{\nu\ast...\ast\nu}_{n}\}_{n\in\mathbb{N}}$.

     We study the asymptotic behavior of the sequence $\{\nu_n\}_{n\in\mathbb{N}}$
     on a finite group, with a complete description of the accumulation points of that sequence, that is, the limit sets of the dynamics $T_{\nu}$. The main point in this note is that our presentation follows a dynamical point of  view, and the main result is obtained with the use of Perron-Frobenius Theorem  (see \cite{PF}).

     We like to point out that our results on the convergence of power are not necessarily new, or a replacement of the classical literature, but just easier to compute and to apply. In the best of our knowledge, there is no direct way to extract this kind of characterization of the limit powers just from the necessary and sufficient, or just sufficient conditions for convergence, that we find on the previous works. More than that, our characterization make use of much more elementary results of analysis and algebra.

     Many of the ideas developed here can be immediately applied to compact (or locally compact) topological groups, but the results will be more abstract, restricted and not computational.

\subsection{Main result}

      In this text we present the following
      \vspace{.3 cm}
\begin{theorem}
Let $G=\{g_0,...,g_{n-1}\}$ be a finite group. If $\nu\in\mathcal{P}(G)$ is an acyclic probability and
$H$ is the subgroup generated by the support of $\nu$, then
$$
\lim_{n\to\infty}\nu^{n}=\sum_{h\in H}\frac{1}{|H|}\delta_h.
$$
\end{theorem}

We also get an interesting result when the probability measure $\nu $ is not acyclic, that is used in the last section in order to obtain a solution for the Choquet-Deny equation.

\section{Proof of the main Theorem  }

We will always denote by $(G=\{g_0, g_1, ..., g_{n-1}\},\cdot)$
a finite group of order $n$ where $g_0 = e$ is the neutral element of the operation $``\cdot"$.
\\
We remind that the space of real continuous functions in $G$, $C(G,\mathbb{R})$ is identified with $\mathbb{R}^{n}$ and we denote
a function $f\in C(G,\mathbb{R})$ by the row vector
$$f(G)=(f(g_0), f(g_1), ..., f(g_{n-1})) \in \mathbb{R}^{n}.$$
As usual, the dual of $C(G,\mathbb{R})$ is identified with $(\mathbb{R}^{n})^{*} \simeq \mathbb{R}^{n}$, is the space of signed measures over $G$,
$$C(G,\mathbb{R})'= \left\{\mu= \sum_{i=0}^{n-1} p_i \delta_{g_{i}}, \; p=(p_0, p_1, ..., p_{n-1}) \in \mathbb{R}^{n} \right\}.$$

in this work we denote
$$\int_{G} f d\mu= \sum_{i=0}^{n-1} p_i f(g_{i})=\langle f(G), p \rangle,$$
where $\langle \cdot , \cdot \rangle$ is the usual scalar product.

In this setting, if $\Delta_n= \{p \in \mathbb{R}^{n} \;|\; p_i \in [0,1],\text{ and } \sum_{i=0}^{n-1} p_i=1\}$ then
$$\mathcal{P}(G)= \left\{\sum_{i=0}^{n-1} p_i \delta_{g_{i}} \in C(G,\mathbb{R})' \; | \; p \in \Delta_n \right \}.$$

If $\nu= \displaystyle\sum_{i=0}^{n-1} p_i \delta_{g_{i}}$ and $ \mu=\displaystyle\sum_{i=0}^{n-1} q_i \delta_{g_{i}}$ we define the convolution between them as
$$(\nu*\mu) (f)=\int_Gfd(\nu\ast\mu)= \int_{G} \int_{G} f(g    h) d\nu(g) d\mu(h).$$

Defining $f(G^2)$ as
$$f(G^2)=
\left[
  \begin{array}{ccc}
    f(g_{0}   g_{0}) & \cdots & f(g_{0}   g_{n-1}) \\
    \vdots & \ddots & \vdots \\
    f(g_{n-1}   g_{0}) & \cdots & f(g_{n-1}   g_{n-1}) \\
  \end{array}
\right]
$$
we get an characterization of the convolution in coordinates.
\begin{lemma}\label{representação bilinear}
If $ \nu=\displaystyle \sum_{i=0}^{n-1} p_i \delta_{g_{i}} \simeq p$ and $ \mu=\displaystyle \sum_{i=0}^{n-1} q_i \delta_{g_{i}} \simeq q$ then
$$(\nu*\mu) (f)=  \langle q,  f(G^2)\cdot p \rangle.$$
\end{lemma}
\begin{proof}
Indeed,
\begin{align*}
(\nu*\mu) (f) & = \int_{G} \int_{G} f(g    h) d\nu(g) d\mu(h)  = \int_{G} \sum_{i=0}^{n-1} p_i f(g_{i}    h)  d\mu(h)\\
& = \displaystyle\sum_{i=0}^{n-1}\displaystyle \sum_{j=0}^{n-1}q_i  p_j  f(g_{i}    g_{j})    = \langle q,  f(G^2) \cdot p\rangle.
\end{align*}

\end{proof}

Since, $\mathcal{P}(G)$ is an affine space of codimension 1 in $C(G,\mathbb{R})'$ we know that $\nu(G\times G)$ is given by an bi-stochastic matrix.
In order to get the next result we define a new matrix obtained by a measure $\nu \simeq (p_0,...,p_{n-1}) \in\mathcal{P}(G)$.
If we denote,
$$G^{-1} \times G=
\left[
  \begin{array}{ccc}
    g_{0}^{-1}   g_{0} & \cdots & g_{0}^{-1}   g_{n-1} \\
    \vdots & \ddots & \vdots \\
    g_{0}^{-1}   g_{n-1} & \cdots & g_{n-1}^{-1}   g_{n-1} \\
  \end{array}
\right],
$$
then
$$\nu(G^{-1} \times G)=
\left[
  \begin{array}{ccc}
    \nu(g_{0}^{-1}   g_{0}) & \cdots & \nu(g_{0}^{-1}   g_{n-1}) \\
    \vdots & \ddots & \vdots \\
    \nu(g_{0}^{-1}   g_{n-1}) & \cdots & \nu(g_{n-1}^{-1}   g_{n-1}) \\
  \end{array}
\right],
$$
where $\nu(g_{i}^{-1} * g_{j})= p_m$ if $g_{i}^{-1} * g_{j}=g_m$.

\begin{lemma}\label{rep. da aplicação conv}
Given $\nu,\mu\in\mathcal{P}(G)$, then
$$\nu * \mu=\mu \cdot \nu(G^{-1} \times G).$$
\end{lemma}

\begin{proof}
If $\nu=\displaystyle \sum_{i=0}^{n-1} p_i \delta_{g_{i}}$ and $ \mu= \displaystyle\sum_{i=0}^{n-1} q_i \delta_{g_{i}}$ we set
$\nu\ast \mu= \displaystyle\sum_{k=0}^{n-1} a_k \delta_{g_{k}}$.

From the Lemma \ref{representação bilinear} we know that,
$$a_k =\displaystyle \sum_{g_{i}    g_{j}= g_k}  p_i q_j =\displaystyle \sum_{i=0}^{n-1}  \{q_i p_j \;|\; g_{j}=g_{i}^{-1}   g_k\}.$$
Since the equation $g_{i}    g_{j}= g_k$ has an unique solution for a fixed $k$ and for each $i$ we have $j(i, k)$ well determined. It allows us to write,
$$a_k = q_0 \cdot p_{j(0, k)} + ... + q_{n-1} \cdot p_{j(n-1, k)}.$$

Using matrices we have
$$
\left[
  \begin{array}{ccc}
    a_{0}  & \cdots & a_{n-1}\\
  \end{array}
\right] =
\left[
  \begin{array}{ccc}
    q_{0}  & \cdots & q_{n-1}\\
  \end{array}
\right] \cdot
\left[
  \begin{array}{ccc}
    p_{j(0, 0)} & \cdots & p_{j(n-1, 0)} \\
    \vdots & \ddots & \vdots \\
    p_{j(0, n-1)} & \cdots & p_{j(n-1, n-1)} \\
  \end{array}
\right].
$$
and we get the first formula
$$\nu * \mu=\mu \cdot \nu(G^{-1} \times G).$$
\end{proof}
Thus, if we desire to compute the powers of the convolution  $\nu\ast\nu$ we have
$$\nu^n:=\underbrace{\nu\ast\ldots\ast\nu}_{m+1}= \nu\cdot\nu(G^{-1} \times G)^{m},$$
so we can estimate the long time behavior of $\nu^n$ from the powers of the matrix  $\nu(G^{-1} \times G)$.

\begin{example}
We consider $G=(\mathbb{Z}_{3}, +)$ and $\nu= (1/3,1/4,5/12)$. So
$$
G^{-1} \times G= \left[
  \begin{array}{ccc}
    0 & 1 & 2 \\
    2 & 0 & 1 \\
    1 & 2 & 0 \\
  \end{array}
\right]\ \mbox{ and }
\nu(G^{-1} \times G)=
\left[
  \begin{array}{ccc}
    1/3 & 1/4 & 5/12 \\
    5/12 & 1/3  & 1/4 \\
    1/4 & 5/12 & 1/3  \\
  \end{array}
\right]
.$$
\end{example}

\begin{definition}
A stochastic matrix $A = (a_{ij})$ is called primitive
 if there is $N \in\mathbb{N}$ such that all the entries of the matrix $A^N$ are positive.
\end{definition}

\begin{definition}
A  matrix A with non-negative entries is called  doubly-stochastic
if its  rows and columns sum 1.
\end{definition}

The following will be very useful in what follows

\begin{theorem}(Perron-Frobenius)\label{teorema da convergencia da matriz}
If A is $n\times n$ primitive and doubly stochastic matrix, then
$$\lim_{m\to\infty}A^{m}=\frac{1}{n}J,$$
where $J = (a_{ij})$, $a_{ij} = 1$ for all i, j.
\end{theorem}

\begin{definition}
Let G be a finite abelian group of order n. We say that $G$ is finitely
generated if there exist $g_{1},...,g_{k}\in G$ such that for all $g\in G$ we have that
 $g=g_{1}^{r_{1}}\cdot\cdot\cdot g_{k}^{r_{k}}$, with $r_{j}\in\{0,1,...,n\}$.
\end{definition}
We remember the definition of the support of a given measure.
Let $G$ a finite group and $\nu=(p_0,...,p_{n-1})\in\mathcal{P}(G)$. The support of $\nu$ is the set
$$\mbox{supp}(\nu)=\{g_i\in G:\nu(g_i)=p_i>0\}.$$
We will denote by $H$ the subgroup of $G$ generated by $\mbox{supp}(\nu)$, i.e., $H=\langle\mbox{supp}(\nu)\rangle$.
In order to get the next result we need a new definition and a Lemma.  We start with the definition:

\begin{definition}\label{Acyclic} (Acyclic)
Given $\nu\in\mathcal{P}(G)$, we define the set $Z_{+}(\nu)^{m}$ by
$$Z_{+}(\nu)^{m}=\{g_{i_1}...g_{i_m}:g_{i_k}\in\mbox{supp}(\nu)\}.$$
Let $H$ the subgroup of $G$ generated by $\mbox{supp}(\nu)$.
We say that $\nu$ is a acyclic probability measure if there exist $N\in\mathbb{N}$ such that
$Z_{+}(\nu)^{N}=H$. In particular, $Z_{+}(\nu)^{1}=\mbox{supp}(\nu)$.
\end{definition}

We would like to observe that, the acyclic property is similar to \cite{Chakraborty} for probabilities in matrices, but in that case the convergence is given by a rank theorem.

\begin{example}
(a) Let $g\in G$ an element of order 2 and $\nu=\delta_g$. In this case $H=\{e,g\}$ and
$$Z_{+}(\nu)^{m}=\left\{
                 \begin{array}{ll}
                   e, & \hbox{ if } m \hbox{ is even} \\
                   g, & \hbox{ if } m \hbox{ is odd.}
                 \end{array}
               \right.$$
From it follows that $\nu$ is not  acyclic probability measure.
\\
(b) Let $H$ a cyclic group generated by $g$ and $\nu=\alpha\delta_e+(1-\alpha)\delta_g$, $0<\alpha<1$.
Then $\nu$ is a acyclic. In fact, if $H=\{e,g,...,g^{n-1}\}$, then
$$Z_{+}(\nu)^{n}=\{e^n,e^{n-1}g,e^{n-2}g^{2},...e^1g^{n-1}\}=H.$$
\end{example}

\begin{example}\label{exemplo de acyclic}
Let G be a finite abelian group of order n and $\nu\in\mathcal{P}(G)$. We can make the identification $\nu =\sum_ip_i\delta_{g_i}\simeq p=(p_0,...,p_{n-1})$. If $Z_{+}(p)=\{g,h\}$
and $H=\langle g^{-1}h\rangle$, then $\nu$ is a acyclic. In fact, to see it we only need to notice that
$$g^{n-k}h^{k}=(g^{-1}h)^{k}.$$
\end{example}

\begin{example}
Let $H=\left \langle g_1,...,g_k\right\rangle$ be a
finitely generated abelian subgroup of $G$ and
$\nu\in\mathcal{P}(G)$.
  If $\nu\in\mathcal{P}(G)$ is such $$Z_{+}(\nu)=\{ e,g_1,...,g_k\}$$
 $\nu$ is a acyclic.
\end{example}

\begin{remark}
If we have that $|G|=n$ and the support
of $\nu$ has more than $ \displaystyle\frac{n}{2}+1$ elements,
then $\nu$ is acyclic, in particular $\nu(G^{-1} \times G)$ primitive.
\end{remark}

When the probability $\nu$ is acyclic we have the following proposition:
\begin{proposition}\label{main theorem}
Let $G=\{g_0,...,g_{n-1}\}$ a finite group, $\nu\in\mathcal{P}(G)$ acyclic and $H=\langle Z _{+}(\nu)\rangle$ . The
matrix $\nu(G^{-1}\times G)=(\nu(g_i^{-1}  g_j))_{i,j}$ satisfies
$\displaystyle \lim_{n\to\infty} \nu(G^{-1}\times G)^{n}=B$, where $B$ is the matrix
given by
\begin{align*}
\left(
  \begin{array}{cccc}
    \frac{1}{|H|}J & 0 &\ldots&0  \\
    0 &  \frac{1}{|H|}J& 0&\ldots \\
    \vdots & \vdots& \vdots &\vdots\\
0&\ldots&0&\frac{1}{|H|}J\\
  \end{array}
\right),
\end{align*}
 where $0$ is the null matrix of order $|H|$ and $J$ is the matrix of order $|H|$, with all the coefficients equal to 1.
\end{proposition}
\begin{proof}
The prove of this result follows form the lemmas below.
\end{proof}
\begin{remark}
Let us consider an acyclic probability $\nu\in\mathcal{P}(G)$,
and the subgroup $H$ generated by $Z_{+}(\nu)$. We suppose that $|H|=n$.
  We can take the equivalence classes determined by $H$ in $G$, i.e,
$$gH=\{g  h:h\in H\}.$$
We know that $G$ can be written as a disjoint union of the equivalence classes determined by $H$,
and as $G$ is finite $H$ is also a finite group. Then we can write $G$ as follows
\begin{align*}
G&=\{e,h_1,...,h_k,g_1h_1,...,g_1,g_1h_k,g_2h_1,...,g_2,g_2h_k,...,g_l,g_lh_1,...,g_lh_k\}
\\&=\{H, g_1H,...,g_lH\},
\end{align*}
where $g_iH\cap g_jH=\emptyset$ for $i\not=j$. Then we have that the matrix $A=\nu(G^{-1}\times G)$ is given by
\begin{align*}
A={\footnotesize\left(
  \begin{array}{cccc}
    \nu(H^{-1}\times H) & \nu(H^{-1}\times g_1H) & \ldots &\nu(H^{-1}\times g_lH) \\
    \nu(H^{-1}g_{1}^{-1}\times H) & \nu(H^{-1}g_{1}^{-1}\times g_1H) & \ldots & \nu(H^{-1}g_{1}^{-1}\times g_l H) \\
    \vdots & \vdots & \vdots & \vdots \\
    \nu(H^{-1}g_{l}^{-1}\times H) & \ldots & \ldots & \nu(H^{-1}g_{l}^{-1}\times g_lH) \\
  \end{array}
\right),}
\end{align*}
where the blocks in the diagonal are always the matrix $\nu(H^{-1}\times H)$.
\end{remark}

\begin{lemma}\label{lema da matriz nula}
The blocks $\nu(H^{-1}g_{i}^{-1}\times g_j H)$, $\nu(H^{-1}g_{i}^{-1}\times H)$
and $\nu(H^{-1}\times g_j H)$  are always the null matrix for $i\not=j$.
\end{lemma}
\begin{proof}
Take the block $\nu(H^{-1}g_{1}^{-1}\times g_2 H)$ and notice that
\begin{align*}
 \nu((  h_{i}^{-1}g_{1}^{-1}) (g_2  h_j))>0
&\Leftrightarrow (  h_{i}^{-1}g_{1}^{-1}) (g_2  h_j)\in Z_{+}(\nu)\subset H
\\&\Rightarrow g_{1}^{-1}  g_2\in H\Rightarrow g_1H=g_2H,
\end{align*}
but it is a contradiction, since $g_1H\cap g_2H=\emptyset$. By analogous computations we have the result for the others cases.
\end{proof}

By Lemma \ref{lema da matriz nula}, we have that the powers of the matrix $\nu(G^{-1}\times G)$ are given by
\begin{align*}
\nu(G^{-1}\times G)^n={\footnotesize\left(
  \begin{array}{cccc}
    \nu(H^{-1}\times H)^n & 0 & \ldots &0 \\
   0 & \nu(H^{-1}\times H)^n & \ldots & 0 \\
    \vdots & \vdots & \vdots & \vdots \\
    0 & \ldots & \ldots & \nu(H^{-1}\times H)^n \\
  \end{array}
\right).}
\end{align*}

\begin{lemma}
The matrix $\nu(H^{-1}\times H)$ is primitive.
\end{lemma}
\begin{proof}
Let us consider the matrix $A=(a_{ij})_{i,j}:=(\nu(h_i^{-1}  h_j))_{i,j}$.
Then we notice that
\begin{align*}
a_{ij}>0&\Leftrightarrow h_i^{-1}  h_j\in Z _{+}(\nu)
\\&\Leftrightarrow\exists \bar{h}\in Z _{+}(\nu) \mbox{ such that }h_i^{-1}  h_j= \bar{h}
\\&\Leftrightarrow h_j= h_i \bar{h},\ \bar{h}\in Z _{+}(\nu).
\end{align*}
It implies that $a_{ij}>0$ if and only if $h_j\in L_{h_{i}}(Z _{+}(\nu))$. As $L_{h_{i}}$ is
a bijection, in each line we have $|Z _{+}(\nu)|$ positive coefficients.
Consider now $A^2$, which we will denote by $A^2=(a_{ij}^{2})_{i,j}$. Then we have that
  \begin{align*}
a_{ij}^{2}>0&\Leftrightarrow \sum_{k=0}^{n-1}\nu(h_i^{-1}  h_k)\nu(h_{k}^{-1}  h_j)
\\&\Leftrightarrow\exists k\in \{0,...,n-1\}\mbox{ such that }\nu(h_i^{-1}  h_k)\nu(h_k^{-1}  h_j)>0
\\&\Leftrightarrow \nu(h_i^{-1}  h_k)>0\mbox{ and } \nu(h_k^{-1}  h_j)>0
\\&\Leftrightarrow \exists h^{\prime},h^{\prime\prime}\in Z _{+}(\nu)\mbox{ such that }h_k=h_{i}  h^{\prime}, \ h_j=h_k  h^{\prime\prime}
\\&\Leftrightarrow h_j=h_i  h^{\prime}  h^{\prime\prime}
\\&\Leftrightarrow h_j\in L_{h_i}(Z _{+}(\nu)^{2}).
\end{align*}
Again, we can see that $A^2$ has $|Z _{+}(\nu)^{2}|$  positive coefficients. Following by induction, if
$A^{n}=(a_{ij}^{n})_{i,j}$, then
$$a_{ij}^{n}>0\Leftrightarrow h_j\in L_{h_i}(Z _{+}(\nu)^{n}).$$
As $\nu$ is acyclic we have, from Definition \ref{Acyclic} that there exists $N\in\mathbb{N}$, such that for $n>N$
$$a_{ij}^{n}>0\Leftrightarrow h_j\in L_{h_i}(Z _{+}(\nu)^{n})=h_i H=H.$$
It implies that for $n>N$ the matrix $A^{n}=(a_{ij}^{n})_{i,j}$ has $|H|$ coefficients positive in each line.
As the matrix $A$ has order $|H|$ we see that $A$ is primitive.
\end{proof}

\begin{lemma}\label{permutation lemma}
Let $\mu,\nu\in\mathcal{P}(G)$ and $\sigma$ a permutation on $G$. Then we have that
$$\mu\cdot\nu((\sigma(G))^{-1}\times \sigma(G))=\mu\cdot\nu(G^{-1}\times G).$$
\end{lemma}
\begin{proof}
We notice that the convolution does not depend on the order of the group, then
\begin{align*}
\mu\cdot\nu((\sigma(G))^{-1}\times \sigma(G))=\mu\ast\nu
=\mu\ast\nu(G^{-1}\times G).
\end{align*}
\end{proof}

 We would like to observe that, $B$ is also doubly stochastic and always has 1 as an eigenvalue.

\begin{remark} The main fact used in the Lemma \ref{permutation lemma} was the fact that the integral does not change under
permutation of the group $G$.
\end{remark}

Using Lemma \ref{permutation lemma} and Proposition \ref{main theorem}, and making some permutation on the elements one can easily conclude that,
\begin{proposition}\label{the w-limit theorem}
Let $G=\{g_0,...,g_{n-1}\}$ be a finite group, $\nu\in\mathcal{P}(G)$ an acyclic probability and $H=\langle Z _{+}(\nu)\rangle$ . The
matrix $\nu(G^{-1}\times G)=(\nu(g_i^{-1}  g_j))_{i,j}$ satisfies
$\displaystyle \lim_{n\to\infty} \nu(G^{-1}\times G)^{n}=B$, where $B$ is the matrix
given by
$$
b_{ij}=\left\{
         \begin{array}{ll}
           0, & \hbox{ if } g_i^{-1}  g_j\not\in H \\
           \frac{1}{|H|}, & \hbox{ if } g_i^{-1}  g_j\in H
         \end{array}
       \right.
$$
\end{proposition}
Applying Proposition \ref{the w-limit theorem} we get the main result:

\begin{theorem}
Let $G=\{g_0,...,g_{n-1}\}$ be a finite group. If $\nu\in\mathcal{P}(G)$ is an acyclic probability
$H=\langle Z _{+}(\nu)\rangle$ is the subgroup generated by the support of $\nu$, then
$$
\lim_{n\to\infty}\nu^{n}=\sum_{h\in H}\frac{1}{|H|}\delta_h.
$$
\end{theorem}
\begin{example} \label{example subgroupZ+} Take $\bar{G}$ a finite abelian group  and $a, b \in \bar{G}$ such that $a^2=e$, $b^3=e$ and $\nu=p=\alpha \delta_e + (1-\alpha) \delta_b$, $0<\alpha<1$  and $G=\{e, a, b^2, ab, b,ab^2\}$.
So we have
$$
\nu(G^{-1}\times G)=\left(
                      \begin{array}{cccccc}
                        \alpha & 0 & 0 & 0 & (1-\alpha) & 0 \\
                        0 & \alpha & 0 & (1-\alpha) & 0 & 0 \\
                        (1-\alpha) & 0& \alpha & 0 & 0 & 0 \\
                        0 & 0 & 0 & \alpha & 0 & (1-\alpha) \\
                        0 & 0 & (1-\alpha) & 0 & \alpha & 0 \\
                        0 & (1-\alpha) & 0 & 0 & 0 & \alpha \\
                      \end{array}
                    \right).
$$
 In that case $Z_{+}(\nu)=\{e,b\}$ and $\langle Z_{+}(\nu)\rangle=\{e,b,b^2\}$, and by Theorem \ref{the w-limit theorem} we have that
$$\lim_{n\to\infty}\nu(G^{-1}\times G)^{n}= \left(
                               \begin{array}{cccccc}
                                 \frac{1}{3} & 0 & \frac{1}{3} & 0 & \frac{1}{3} & 0 \\
                                 0 & \frac{1}{3} & 0 & \frac{1}{3} & 0 & \frac{1}{3} \\
                                 \frac{1}{3} & 0 & \frac{1}{3} & 0 & \frac{1}{3} & 0 \\
                                 0 & \frac{1}{3} & 0 & \frac{1}{3} & 0 & \frac{1}{3} \\
                                 \frac{1}{3} & 0 & \frac{1}{3} & 0 & \frac{1}{3} & 0 \\
                                 0 & \frac{1}{3} & 0 & \frac{1}{3} & 0 & \frac{1}{3} \\
                               \end{array}
                             \right).
$$
Then we have that $\displaystyle\lim_{n\to\infty}\nu^n=\frac{1}{3}(\delta_e+\delta_b+\delta_{g^2})$.
\end{example}
\begin{remark}
If the probability $\nu\in\mathcal{P}(G)$ is not acyclic,
then there exists a finite number of subsets of $\langle Z_{+}(\nu)\rangle$,
let us say $K_1,...,K_l$ such that for each $n\in\mathbb{N}$, $Z_{+}(\nu)^n=K_i$
for some $i\in\{1,...,l\}$. Following the same computations
made to get Theorem \ref{main theorem}, is possible to show that the sequence
$\{\nu^n\}_{n\in\mathbb{N}}$ has $l$ accumulation points and each of these accumulations
points is a uniform probability measure supported on a set $K_j$.
\end{remark}

\begin{remark}
Let $G_1,G_2$ finite groups and $\phi:G_1\to G_2$
a homomorphism of groups. It is easy to see that the
push forward map
$\phi_{\sharp}:\mathcal{P}(G_1)\to \mathcal{P}(G_2)$
given by $\phi_{\sharp}(\mu)(A)=\mu(\phi^{-1}(A))$ for all
$A\subset G_2$,
satisfies the following
$$
\phi_{\sharp}(\nu\ast\mu)=\phi_{\sharp}(\nu)\ast\phi_{\sharp}(\mu).
$$
It implies that $\lim_{n\to\infty}\phi_{\sharp}(\nu^n)=\lim_{n\to\infty}(\phi_{\sharp}(\nu))^n$.
\end{remark}

The next proposition guarantee the density of the set of acyclic probability.
It shows how big is the set of acyclic probabilities in the sense of the topology
of  $\mathcal{P}(G)$.

\begin{proposition}
Let $\nu_0\in\mathcal{P}(G)$, where $G$ is a finite group. Given $\varepsilon>0$, there exists $\bar{\nu}\in\mathcal{P}(G)$
such that $\bar{\nu}$ is a acyclic and $d(\bar{\nu},\nu_0)<\varepsilon$, i.e., the set of acyclic probabilities
is  dense in $\mathcal{P}(G)$.
\end{proposition}
\begin{proof}
Let $\varepsilon>0$ and $\nu_0=p=\displaystyle\sum_{i=0}^{k-1}p_i\delta_{h_i}$ with
$$Z_{+}(p)=\{g\in G:\nu(g)>0\} \mbox{ and } H=\langle Z_{+}(p)\rangle=\{h_0,...,h_{k-1}\}.$$
 Then we define $a=\min\{p_i:p_i>0\}$ and  $\displaystyle\bar{\varepsilon}=\frac{1}{2}\min\{\varepsilon,a \}$.
So we consider the measure $\displaystyle\bar{\nu}=\bar{p}=\sum_{i=0}^{k-1}\bar{p}_i\delta_{h_i}$, where
$$
\displaystyle\bar{p}_i=\left\{
  \begin{array}{ll}
   \frac{ \bar{\varepsilon}}{k-|Z_{+}(p)|}, & \hbox{ if }p_i=0 \\
    p_i-\frac{ \bar{\varepsilon}}{|Z_{+}(p)|}, & \hbox{ if } p_i>0.
  \end{array}
\right.
$$
Obviously $\bar{\nu}\in\mathcal{P}(G)$ and as
$$
d(\nu_0,\bar{\nu})=\sum_i|p_i-\bar{p}_i|=\sum_{p_i=0}\frac{ \bar{\varepsilon}}{k-|Z_{+}(p)|}+\sum_{p_i>0}\frac{ \bar{\varepsilon}}{|Z_{+}(p)|}=2\bar{\varepsilon}<\varepsilon,
$$
so we get the result.
\end{proof}

\section{Application: Dynamics of $T_{\nu}$}

We start this section with the basic properties of $T_{\nu}$ .

\begin{proposition}
The map $T_{\nu}(\mu)= \mu*\nu$ is continuous in the weak topology, linear and its fixed points satisfy the Choquet-Deny equation, $\mu*\nu=\mu$.
\end{proposition}

This claims are long time knowledged from the literature (see \cite{Bob} pg 73,  \cite{CD} and \cite{D}), thus what remains is to understand the asymptotic behavior of $T_{\nu}$.

\begin{theorem}\label{teoremdynlim} Let $G=\{g_0,...,g_{n-1}\}$ be a finite group. If $\nu\in\mathcal{P}(G)$ is an acyclic probability, $H=\langle Z _{+}(\nu)\rangle$ is the subgroup generated by the support of $\nu$, the
$w-$limit set, here denoted by $L_{\omega}(\mu)$, that is the set of accumulation point of his orbit, is
$$L_{\omega}(\mu)= \sum_{h\in H}\frac{1}{|H|}\delta_h   * \mu,$$
  linear on $\mu$. Moreover, $\mu$ is a recurrent point of the dynamics, that is, $\mu \in L_{\omega}(\mu)$, only if $\mu$ is solution of the Choquet-Deny equation
$$\bar{\nu}*\mu=\mu,$$
where $\displaystyle \bar{\nu}=\lim_{n\to\infty}\nu^{n}$.
\end{theorem}

\begin{example} \label{exdyn}

We consider $G=(\mathbb{Z}_{3}, +)$ and $\nu= (1/3,1/4,5/12)$. So
$$
G^{-1} \times G= \left[
  \begin{array}{ccc}
    0 & 1 & 2 \\
    2 & 0 & 1 \\
    1 & 2 & 0 \\
  \end{array}
\right]\ \mbox{ and }
\nu(G^{-1} \times G)=
\left[
  \begin{array}{ccc}
    1/3 & 1/4 & 5/12 \\
    5/12 & 1/3  & 1/4 \\
    1/4 & 5/12 & 1/3  \\
  \end{array}
\right]
.$$

To find the fixed points for $T_{\nu}$ we need to solve the following equation,
$$\left[
  \begin{array}{ccc}
    q_{0}  & q_{1} & q_{2}\\
  \end{array}
\right] =
\left[
  \begin{array}{ccc}
    q_{0}  & q_{1} & q_{2}\\
  \end{array}
\right] \cdot
 \left[
  \begin{array}{ccc}
    1/3 & 1/4 & 5/12 \\
    5/12 & 1/3  & 1/4 \\
    1/4 & 5/12 & 1/3  \\
  \end{array}
\right].$$
By linear algebra we have that there is only one solution for the above
equation and it is given by $\displaystyle\mu_0=\frac{1}{3}(\delta_0+\delta_1+\delta_2)$.
So the unique fixed point is $\mu_0$.

We also have that $\mu$ is recurrent only if $\displaystyle\lim_{n\to\infty}\nu^{n}\ast\mu=\mu$.
But $\displaystyle\lim_{n\to\infty}\nu^{n}=\mu_0$, and it implies that $\mu$ is recurrent only if
$$
\left[
  \begin{array}{ccc}
    q_{0}  & q_{1} & q_{2}\\
  \end{array}
\right] =
\left[
  \begin{array}{ccc}
    q_{0}  & q_{1} & q_{2}\\
  \end{array}
\right] \cdot
 \left[
  \begin{array}{ccc}
    1/3 & 1/3 & 1/3 \\
    1/3 & 1/3  & 1/3 \\
    1/3 & 1/3 & 1/3  \\
  \end{array}
\right],
$$
and solving this equation the unique possibility is $\mu=\mu_0$.
\end{example}

Using Theorem \ref{main theorem}, we will try to find conditions
for two measures have the same $\omega-$limit, where $\nu=\sum_{i}p_i\delta_{g_{i}}$ is a acyclic. First we observe that
if $\mu=\sum_{i}q_i\delta_{g_{i}}\in\mathcal{P}(G)$, then $\omega(\mu)=\{\mu\cdot B\}$. If we identify $\mu$ with
the vector $q=\sum_{i}q_ie_i$ in $\mathbb{R}^{n}$, where $\{e_i\}_{0\leq i\leq n-1}$ is the canonical basis of $\mathbb{R}^{n}$, we have that
\begin{align*}
q\cdot B=\Big(\sum_{i}q_ie_i\Big)\cdot B=\sum_{i}q_i\Big(e_i\cdot B\Big).
\end{align*}
It implies that $L_{\omega}(\mu)=\sum_{i}q_i\omega(\delta_{g_i})$. So, to determine the $\omega-$limit of a measure it is
enough to determine the $\omega-$limit of the measures $\delta_{g_i}$, for all $g_i\in G$. Then we notice that if
$H=\langle Z_{+}(p)\rangle$, $|H|=k$, $|G|=|H|l$, $\bar{\mu}=(q_0,...,q_{n-1})$, $\mu=\delta_{g_{0}}$, and if we write $$\alpha_0=\displaystyle\sum_{i=0}^{k-1}q_i,\
 \alpha_1=\displaystyle\sum_{i=k}^{2k-1}q_i,..., \ \displaystyle\alpha_l=\sum_{i=n-k-1}^{n-1}q_i$$
\begin{align*}
 \mu\cdot B&=\bar{\mu}\cdot B
\\&\Leftrightarrow
\Big(\underbrace{\frac{1}{k},\frac{1}{k},...,\frac{1}{k}}_{k},0,...,0\Big)=
\Big(\underbrace{\frac{1}{k}\alpha_0,...,\frac{1}{k}\alpha_0}_{k},\underbrace{\frac{1}{k}\alpha_1,...,
\frac{1}{k}\alpha_1}_{k},...,\underbrace{\frac{1}{k}\alpha_l,...,\frac{1}{k}\alpha_l}_{k}\Big)
\\&\Leftrightarrow
    \displaystyle\sum_{i=0}^{k-1}q_i=1 , \
    \displaystyle\sum_{i=k}^{2k-1}q_i =0,\
   ...., \
\sum_{i=n-k-1}^{n-1}q_i=0.
  \end{align*}
It implies that $L_{\omega}(\delta_{g_0})=L_{\omega}(\mu)$, if and only if $\displaystyle\sum_{i=0}^{k-1}q_i=1$, where $\mu=(1,0,...,0)$.
By the same argument used above we can see that
 \begin{align*}
&L_{\omega}(\delta_{g_i})=L_{\omega}(\delta_{g_0}) \mbox{ for } 0\leq i \leq k-1, \
L_{\omega}(\delta_{g_i})=L_{\omega}(\delta_{g_k}) \mbox{ for } k\leq i \leq 2k-1,...,
\\&L_{\omega}(\delta_{g_i})=L_{\omega}(\delta_{g_{n-k-1}}) \mbox{ for } n-k-1\leq i \leq n-1,
\end{align*}
and from it follows that $\displaystyle L_{\omega}(\mu)=\sum_{i}q_iL_{\omega}(\delta_{g_i})=\sum_{j=0}^{l}\alpha_{j}L_{\omega}(\delta_{g_{jk}})$,
and if $\bar{\mu}=(q_0,...,q_{n-1})$ and we take $\mu=\delta_{g_i}$, with $mk\leq i\leq mk-1$,
$$
L_{\omega}(\bar{\mu})=L_{\omega}(\delta_{g_i})\Leftrightarrow\displaystyle\alpha_{m}=\sum_{i=mk}^{mk-1}q_i=1,
\mbox{ and } \alpha_j=0 \mbox{ for }j\not=m.
$$
Finally, given $\mu=(q_0,...,q_{n-1})$ and $\mu^{\prime}=(q_{0}^{\prime},...,q_{n-1}^{\prime})$,
$$L_{\omega}(\mu)=L_{\omega}(\mu^{\prime})\Leftrightarrow \displaystyle\sum_{i=0}^{k-1}q_i=\displaystyle\sum_{i=0}^{k-1}q_i^{\prime},\
 \displaystyle\sum_{i=k}^{2k-1}q_i=\displaystyle\sum_{i=k}^{2k-1}q_i^{\prime},..., \ \displaystyle\sum_{i=n-k-1}^{n-1}q_i=\sum_{i=n-k-1}^{n-1}q_i^{\prime}.
$$

\begin{definition}
Let $\nu\in\mathcal{P}(G)$ a acyclic probability measure and $\eta\in\mathcal{P}(G)$. We call the basin of $\eta$ the set
$$\{\mu\in\mathcal{P}(G):\lim_{n\to\infty}T_{\nu}^{n}(\mu)=\eta\}.$$
\end{definition}
\begin{example}
Let's go back to the Example \ref{example subgroupZ+} where $G=\{e, a, b^2, ab, b,ab^2\}$, in that particular situation, $\nu=p=\alpha \delta_e + (1-\alpha) \delta_b$, $0<\alpha<1$ and we can rewrite $G$ as $G=\{e,b , b^2, a, ab, ,ab^2\}$ as in Lemma \ref{lema da matriz nula}.

Then , given $\mu=(q_0,...,q_{5})$ and $\mu^{\prime}=(q_{0}^{\prime},...,q_{5}^{\prime})$, we have
$$L_{\omega}(\mu)=L_{\omega}(\mu^{\prime})\Leftrightarrow \displaystyle\sum_{i=0}^{2}q_i=\displaystyle\sum_{i=0}^{2}q_i^{\prime}\text{, and  } \displaystyle\sum_{i=3}^{5}q_i=\sum_{i=3}^{5}q_i^{\prime}.
$$

For instance, if $\mu^{\prime}=(\frac{1}{4}, \frac{1}{2}, 0, \frac{1}{8}, 0, \frac{1}{8})$ we have
\begin{align*}
 \lim_{n\to\infty}T_{\nu}^{n}(\mu^{\prime}) &=\frac{1}{3}\Big(q_0^{\prime}+q_1^{\prime}+q_2^{\prime}, ...,q_0^{\prime}+q_1^{\prime}+q_2^{\prime}, q_3^{\prime}+q_4^{\prime}+q_5^{\prime}, ..., q_3^{\prime}+q_4^{\prime}+q_5^{\prime} \Big) \\
 &= \left(\frac{1}{4},\frac{1}{4},\frac{1}{4}, \frac{1}{12}, \frac{1}{12},\frac{1}{12}\right)\\
 &=\omega(\mu^{\prime})= \eta.
\end{align*}

So, the basin of attraction of $\eta=\left(\frac{1}{4},\frac{1}{4},\frac{1}{4}, \frac{1}{12}, \frac{1}{12},\frac{1}{12}\right)$, that is, $$\{\mu=(q_0,...,q_{5}) \; | \; \lim_{n\to\infty}T_{\nu}^{n}(\mu)=  \eta\}$$ is given by
\[\left\{
  \begin{array}{ll}
    q_0 + q_1 + q_{2} &= \frac{3}{4} \\
    q_3 + q_4 + q_{5} &= \frac{1}{4}\\
    q_0, ... , q_{5} & \in [0,1]
  \end{array}
\right.
\]
that  is a convex region of hyperplane in $\mathbb{R}^{6}$ of dimension 4, more precisely
\[
\left\{
  \begin{array}{l}
    q_0 =  \frac{3}{4} - a - b\\
    q_1 = a ,  q_{2}= b \\
    q_3=\frac{1}{4} - c - d\\
    q_4= c,  q_{5}= d\\
    a+ b \leq \frac{3}{4}, \; c+d \leq \frac{1}{4} \\
    a, b, c, d  \in [0,1]\\
  \end{array}
\right.
\]
is the basin of attraction of $\eta=\left(\frac{1}{4},\frac{1}{4},\frac{1}{4}, \frac{1}{12}, \frac{1}{12},\frac{1}{12}\right)$.
\end{example}
Actually the next theorem shows the this always happens.
\begin{proposition}
 Let $\nu=p\in\mathcal{P}(G)$ a acyclic
 probability measure and $H=\langle Z_+(p)\rangle$, with $|H|=k$ and $|G|=|H|l$. Given $\eta\in\mathcal{P}(G)$
 with
 $$\eta=(\underbrace{q_0,...,q_0}_k,\underbrace{q_1,...,q_1}_k,...,\underbrace{q_{l-1},...,q_{l-1}}_k)$$
  The basin of $\eta$ is a convex subset of a hyperplane of dimension $\frac{n(k-1)}{k}$ in $\mathbb{R}^{n}$.
\end{proposition}
\begin{proof}
To prove the convexity of the basin of a given $\eta$ we only need to notice
that if $\mu_1,\mu_2\in\mathcal{P}(G)$ and $0\leq\alpha\leq1$ , then
\begin{align*}
T_{\nu}(\alpha \mu_1+(1-\alpha)\mu_2)&=(\alpha \mu_1+(1-\alpha)\mu_2)\cdot\nu(G^{-1}\times G)
\\&= \alpha\mu\cdot\nu(G^{-1}\times G)+(1-\alpha)\mu_2\cdot\nu(G^{-1}\times G).
\end{align*}
Hence, if $\mu_1,\mu_2$ are in the basin of $\eta$, the
\begin{align*}
\lim_{n\to\infty}T^{n}_{\nu}(\alpha \mu_1+(1-\alpha)\mu_2)&=\alpha \lim_{n\to\infty}T^{n}_{\nu}(\mu_1)+(1-\alpha)\lim_{n\to\infty}T^{n}_{\nu}(\mu_2)
\\&=\alpha\eta+(1-\alpha)\eta=\eta.
\end{align*}
To  prove the second part of the theorem we notice that if
$\mu=(q^{\prime}_0,q^{\prime}_1,...,q^{\prime}_{n-1})$ is in the basin of $\eta$, then
$\lim_{n\to\infty}T^{n}_{\nu}(\mu)=\eta$ if and only if
\begin{equation*}
\left\{
                                           \begin{array}{ll}
                                             \displaystyle\frac{1}{k}\sum_{i=0}^{k-1}q_i^{\prime}=q_0 \\
                                             \ \ \ \vdots \\
                                             \displaystyle\frac{1}{k}\sum_{i=n-k-1}^{n-1}q_i^{\prime}=q_{l-1}\\
q_0,...,q_{n-1}\in[0,1].
                                           \end{array}
                                         \right.
\end{equation*}
If we forget the restriction $\displaystyle\sum_{i=0}^{n-1}q_i^{\prime}=1$
and $q_0,...,q_{n-1}\in[0,1]$, we have a linear system
which has $n$ variables and $l$ linearly independents equations.
Then its space of solution is given by an hyperplane of dimension
$$n-l=n-\frac{n}{k}=\frac{n(k-1)}{k}.$$
Then the solution of system above is a convex set given by the intersection of a hyperplane
of dimension $\frac{n(k-1)}{k}$
 with the simplex
$\Delta_n=\{(x_0,...,x_{n-1}):\sum_i x_i=1, \ x_i\in[0,1] \}$.
\end{proof}

Thus we have a complete characterization of the limit set of this dynamics, but the generical behaviour of this dynamical systems is given by Example~\ref{exdyn}. Indeed we can prove the following:
\begin{theorem}\label{genericbeh} There is an open and dense set $\mathcal{O} \subset \mathcal{P}(G)$ such that, for all $\nu\in \mathcal{O}$,
$$L_{\omega}(\mu)= \{\nu_0 * \mu\}, \forall \mu\in\mathcal{P}(G),$$
where $\displaystyle \nu_0=\frac{1}{|G|}\sum_{i=0}^{|G|-1}\delta_{g_i}$ which is the unique fixed point of $T_{\nu}$.
\end{theorem}
\begin{proof}  Consider the set,
$$\mathcal{O}=\{\nu=\sum_{i=0}^{|G|-1} p_{i} \delta_{g_i} \in \mathcal{P}(G) \; | \; p_{i}>0\}.$$
Thus, for $T_{\nu}$ with $\nu\in \mathcal{O}$, $\nu$ is an acyclic probability $H=G$, and Theorem~\ref{teoremdynlim}   claims that the $w-$limit set by $T_{\nu}$, here denoted by $L_{\omega}(\mu)$, is
$$L_{\omega}(\mu)= \frac{1}{|G|}\sum_{i=0}^{|G|-1}\delta_{g_i}   * \mu.$$
Moreover, $\mu$ is a recurrent point of the dynamics, that is, $\mu \in L_{\omega}(\mu)$, only if $\mu$ is solution of the Choquet-Deny equation $$\nu_0*\mu=\mu,$$
where $\displaystyle  \nu_{0}=\lim_{n\to\infty}\nu^{n}= \frac{1}{|G|}\sum_{i=0}^{|G|-1}\delta_{g_i}$.
From the theory of doubly stochastic matrices, we know that the only solution is $\mu=\nu_{0}$, since every fixed point is recurrent, there is just one of them. Thus, we just need to prove that  $\mathcal{O}$ is an open and dense set, but it is trivial because its complementary set is
$$\mathcal{O}^{c}=\{\nu=\sum_{i=0}^{|G|-1} p_{i} \delta_{g_i} \in \mathcal{P}(G) \; | \; \exists p_{i}=0\},$$  is an finite union algebraic sets in $\mathbb{R}^{|G|}$,  so it is closed with empty interior, what conclude the proof.
\end{proof}


\end{document}